\documentclass[11pt]{amsart}
\usepackage{amsfonts,latexsym,amsthm,amssymb,graphicx}
\usepackage[all]{xy}
\usepackage{hyperref}
\usepackage[usenames]{color} 

\DeclareFontFamily{OT1}{rsfs}{}
\DeclareFontShape{OT1}{rsfs}{n}{it}{<-> rsfs10}{}
\DeclareMathAlphabet{\mathscr}{OT1}{rsfs}{n}{it}

\CompileMatrices

\newtheorem{theorem}{Theorem}[section]

\newtheorem{prop}[theorem]{Proposition}
\newtheorem{claim}[theorem]{Claim}
\newtheorem{conj}{Conjecture}

{\theoremstyle{definition} }
{\theoremstyle{remark} \newtheorem{remark}[theorem]{Remark}
\newtheorem{example}[theorem]{Example}}

\newcommand{\Pbb}{{\mathbb{P}}}
\newcommand{\Rbb}{{\mathbb{R}}}
\newcommand{\cL}{{\mathscr L}}
\newcommand{\cO}{{\mathscr O}}
\newcommand{\ua}{\underline a}
\newcommand{\ue}{\underline e}
\newcommand{\uf}{\underline f}
\newcommand{\ut}{\underline t}
\newcommand{\uv}{\underline v}
\newcommand{\uX}{\underline X}
\newcommand{\Til}[1]{{\widetilde{#1}}}
\newcommand{\qede}{\hfill$\lrcorner$}

\DeclareMathOperator{\Vol}{Vol}

\title{
Segre classes of monomial schemes
}
\author{Paolo Aluffi}
\address{
Mathematics Department, 
Florida State University,
Tallahassee FL 32306, U.S.A.
}
\email{aluffi@math.fsu.edu}

\begin{document}

\begin{abstract}
We propose an explicit formula for the Segre classes of monomial
subschemes of nonsingular varieties, such as schemes defined by
monomial ideals in projective space.  The Segre class is expressed as
a formal integral on a region bounded by the corresponding Newton
polyhedron. We prove this formula for monomial ideals in two variables
and verify it for some families of examples in any number of
variables.
\end{abstract}

\maketitle


\section{Introduction}\label{intro}

\subsection{}
The {\em excess numbers\/} of a subscheme~$S$ of projective space $\Pbb^n$ are 
roughly defined as the numbers of points of intersection in the complement of $S$ 
of $n$ general hypersurfaces of given degrees containing $S$. Many challenging open
enumerative problems, such as the problem of computing characteristic numbers for 
families of plane curves, may be stated in terms of excess numbers. Recently, the problem
of computing excess numbers has been raised in algebraic statistics and in view of
applications to machine learning and ideal regression.

The excess numbers of a subscheme $S$ may be computed from the push-forward of the 
{\em Segre class\/} $s(S,\Pbb^n)$ to~$\Pbb^n$. 
Segre classes are defined for arbitrary closed embeddings of schemes, and in a sense carry
all the intersection-theoretic information associated with the embedding (\cite{85k:14004}, 
Chapters~4 and~6). Thus, they
provide a general context applying in particular to the computation of excess numbers, 
and relating this problem directly with the well-developed tools of Fulton-MacPherson
intersection theory. On the other hand, the computation of 
Segre classes is challenging, and indeed the connection with excess numbers appears 
to have mostly been exploited in the reverse direction---providing algorithms for the 
computation of Segre classes starting from the explicit solution of enumerative problems 
by computer
algebra systems such as Macaulay2 (\cite{M2}). This strategy informs the author's 
implementation of an algorithm for Chern and Segre classes of subschemes of projective 
space in \cite{MR1956868}, 
as well as more recent work on algorithmic computations of these classes 
(\cite{MR2736356}, \cite{EJP}). As is usually the case with applications of computer 
algebra, such methods provide very useful tools for experimentations in small dimension, 
but do not lead to general results.

In this note we conjecture a general formula for the Segre class of a {\em monomial\/}
subscheme, in terms of a corresponding Newton polyhedron. The monomial case is
of independent interest, and in principle more general situations can be reduced to the 
monomial case by means of algebraic homotopies (\cite{arXiv1212.2249}). We prove 
the formula in the case of monomials in two variables in any nonsingular variety,
and verify it for some nontrivial examples in arbitrarily many variables.
The formula is expressed as a formal integral over the region bounded by a Newton 
polyhedron associated with the subscheme. This integral can be computed directly from
a subdivision of the region into simplices. 

\subsection{}
We now state the proposed formula precisely. Let $V$ be a nonsingular variety, and let $X_1,
\dots, X_n$ be nonsingular hypersurfaces meeting with normal crossings in $V$.
For $I=(i_1,\dots,i_n)$, we denote by $X^{I}$ the hypersurface obtained
by taking $X_j$ with multiplicity $i_j$, and call this hypersurface a `monomial' (supported
on $X_1,\dots, X_n$). A {\em monomial subscheme\/} $S$ of $V$ is an intersection of 
monomials $X^{I_k}$ supported on a fixed set of hypersurfaces. The exponents $I_k$
determine a (possibly unbounded) region $N$ in the orthant $\Rbb_{\ge 0}^n$ in $\Rbb^n$,
namely, the complement of the convex hull of the union of the orthants with origins
translated at $I_k$. We call this region the {\em Newton region\/} for the exponents $I_k$.

\begin{example}\label{introex}
For $n=2$ and monomials
\[
X^{(2,6)},\quad
X^{(3,4)},\quad
X^{(4,3)},\quad
X^{(5,2)},\quad
X^{(7,0)},\quad
\]
the Newton region $N$ is as in the following picture:
\begin{center}
\includegraphics[scale=.4]{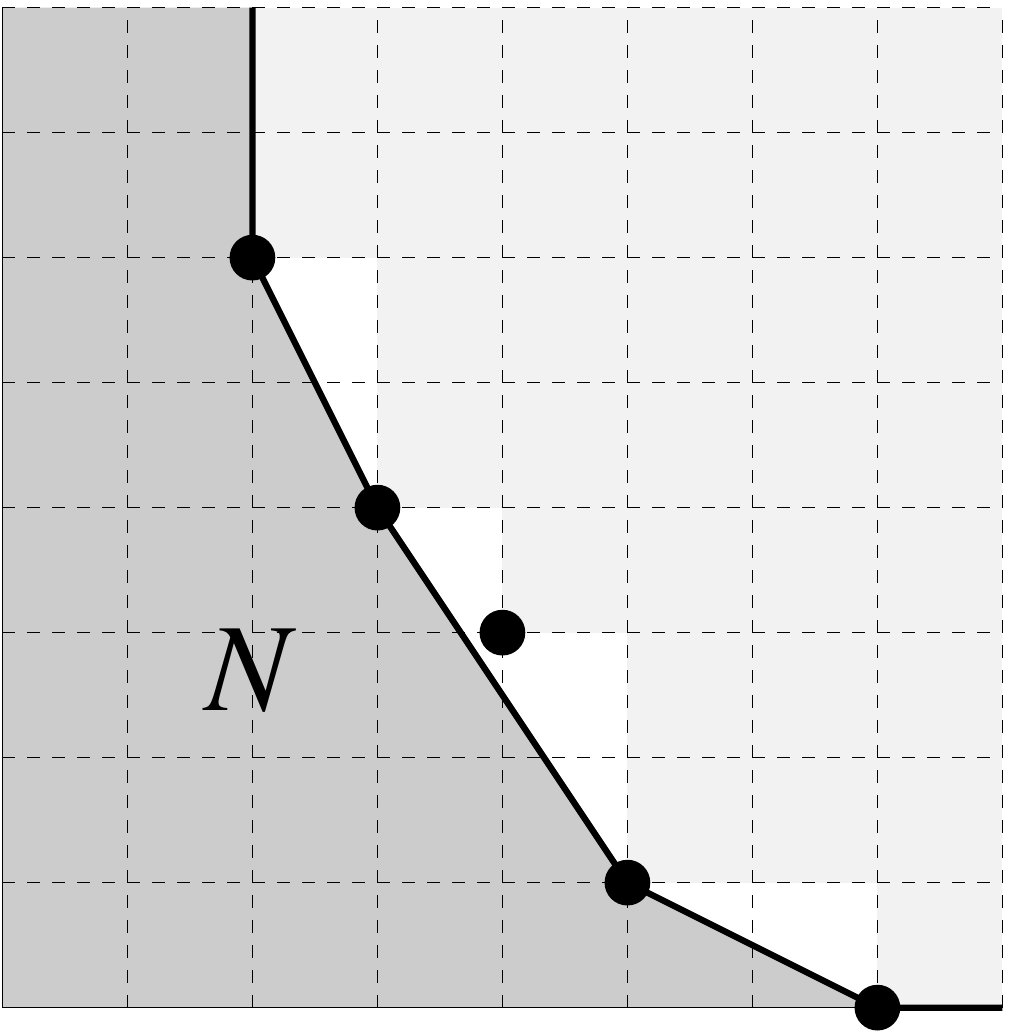}
\end{center}
The third monomial $X^{(4,3)}$ does not affect the Newton region, as it
is contained in the convex hull of the other translated quadrants. (Cf.~Remark~\ref{intdep}.)
\qede\end{example}

\begin{conj}\label{conj}
Let $\iota$ be the inclusion morphism $S\hookrightarrow V$. Then
\begin{equation}\label{eq:main}
\iota_* s(S,V)=\int_N \frac{n! X_1\cdots X_n\, da_1\cdots da_n}
{(1+a_1 X_1+\cdots +a_n X_n)^{n+1}}
\end{equation}
\end{conj}
The right-hand side is interpreted by evaluating the integral formally with $X_1,\dots,X_n$ 
as parameters; the result is a rational function in $X_1,\dots,X_n$, with a well-defined
expansion as a power series in these variables. The claim in Conjecture~\ref{conj} is that 
evaluating the terms of this series as intersection products of the corresponding divisor 
classes in $V$ gives the push-forward $\iota_* s(S,V)$. 

\begin{example}\label{introex2}
Using Fubini's theorem (!)~to perform the integral for the monomials in Example~\ref{introex}
and taking $X_1=X_2=H$ (for example, the hyperplane class in projective space) gives
\begin{multline*}
\int_0^2 \left(\int_0^\infty \frac{2\, H^2\, da_2}{(1+(a_1+a_2)H)^3}\right) da_1
+\int_2^3 \left(\int_0^{10-2a_1} \frac{2\, H^2\, da_2}{(1+(a_1+a_2)H)^3}\right) da_1 \\
+\int_3^5 \left(\int_0^{\frac {17}2-\frac {3a_1}2} \frac{2\, H^2\, da_2}{(1+(a_1+a_2)H)^3}\right) da_1
+\int_5^7 \left(\int_0^{\frac 72-\frac {a_1}2} \frac{2\, H^2\, da_2}{(1+(a_1+a_2)H)^3}\right) da_1\\
=\frac{2H}{1+2H}+
\frac{10H^2(1+5H)}{(1+2H)(1+3H)(1+7H)(1+8H)}+
\frac{2H^2(5+27H)}{(1+3H)(1+5H)(1+6H)(1+7H)} \\
+\frac{2H^2}{(1+5H)(1+6H)(1+7H)}
=\frac{2H(1+30H+168H^2)}{(1+6H)(1+7H)(1+8H)}\quad.
\end{multline*}
See Example~\ref{introex3} below for an alternative way to evaluate this integral.
Expanding as a power series,
\[
\frac{2H(1+30H+168H^2)}{(1+6H)(1+7H)(1+8H)}
=2H+18H^2-334H^3+3714H^4-35278H^5+\cdots
\]
According to Conjecture~\ref{conj}, the scheme $S$ defined by the monomial ideal
\[
I=(x_1^2 x_2^6\,,\,
x_1^3 x_2^4\,,\,
x_1^4 x_2^3\,,\,
x_1^5 x_2\,,\,
x_1^7)
\]
in e.g., $\Pbb^5$ has Segre class
\[
\iota_* s(S,\Pbb^5) = 2[\Pbb^4] + 18 [\Pbb^3] - 334 [\Pbb^2]+ 3714 [\Pbb^1] - 35278 [\Pbb^0]\quad.
\]
This agrees with the output for $I$ of the Macaulay2 procedure computing Segre 
classes given in \cite{MR1956868}:
{\tt\begin{verbatim}
Macaulay2, version 1.4
with packages: ConwayPolynomials, Elimination, IntegralClosure, LLLBases,
               PrimaryDecomposition, ReesAlgebra, TangentCone
i1 : load("CSM.m2");
i2 : QQ[x0,x1,x2,x3,x4,x5];
i3 : time segre ideal (x1^2*x2^6,x1^3*x2^4,x1^4*x2^3,x1^5*x2,x1^7)
                      5        4       3      2
Segre class : - 35278H  + 3714H  - 334H  + 18H  + 2H
     -- used 44.1409 seconds
\end{verbatim}}
In terms of intersection numbers, the equivalence of the subscheme $S$ defined
by $I$ in the intersection of hypersurfaces of degree $d_1,\dots,d_5$
is the coefficient of $H^5$ in
\begin{equation}\label{eq:excessexample}
\left(\prod_{i=1}^5 (1+d_i H)\right)\cdot (2H+18H^2-334H^3+3714H^4-35278H^5)\quad,
\end{equation}
provided that the hypersurfaces cut out $S$ in a neighborhood of $S$. The excess number 
is the difference between this number and the B\'ezout number $d_1\cdots d_5$.
\qede\end{example}

\subsection{}
In this note we prove:

\begin{theorem}\label{mainthm}
Conjecture~\ref{conj} holds for $n\le 2$.
\end{theorem}

The proof of Theorem~\ref{mainthm} is a direct application of techniques in Fulton-MacPherson
intersection theory. Concerning the validity of \eqref{eq:main} for $n>2$, complete intersections 
of monomials $x_1^{m_1},\dots, x_n^{m_n}$ give a straightforward, but maybe too simple 
example (\S\ref{ci}). We give more evidence 
for this formula in terms of a family of nontrivial examples for arbitrary $n$, namely the monomial 
ideals corresponding to the exponents
\[
(0,1,\dots,1),\quad (1,0,1,\dots,1),\quad \dots, \quad (1,\dots, 1,0)\quad.
\]
For these examples we can compute independently the Segre class using the relation 
between Segre classes of singularity subschemes and Chern-Schwartz-MacPherson
classes (\cite{MR2001i:14009}), and we find (Proposition~\ref{nontri}) that the expression 
we obtain does match the result of applying the formula given in Conjecture~\ref{conj}. 

\subsection{}
Excess numbers of monomial ideals admit an expression in terms of mixed volumes of
polytopes, via Bernstein's theorem; the example of the monomial ideal 
$(x_1^{p_1},\dots,x_k^{p_k})$ is worked out explicitly in~\cite{arXiv1212.2249}. Thus, the 
expression obtained in~\eqref{eq:excessexample} may be interpreted as a computation 
of the mixed volumes of certain polytopes in terms of the integral in~\eqref{eq:main}, for
the monomial subscheme of Example~\ref{introex}. Conversely, Bernstein's theorem 
may offer a path to the proof of the conjecture stated above for $n>2$, at least if the 
classes $X_i$ are ample enough. We do not pursue this approach here; Bernstein's
theorem is not used in our proof of Theorem~\ref{mainthm}. 

A precise relation between Segre classes, volumes of convex bodies, and integrals such 
as those appearing in Conjecture~\ref{conj} would be very valuable. Formula \eqref{eq:main} 
(if verified) suggests that the Segre class of the scheme defined by an ideal $I$ may be
computed as a suitable integral over a region in $\Rbb_{\ge 0}^n$ associated with $I$.
The natural guess is that the convex bodies appearing in the work of Lazarsfeld and 
Musta{\c{t}}{\u{a}} (\cite{MR2571958}) and Kaveh and Khovanskii (\cite{MR2950767}) 
would play a key role in such a  result.

As mentioned above, current algorithms for Segre classes essentially reduce the 
computation to enumerative problems, which are then solved by methods in computer
algebra. This limits substantially the scope of these algorithms, and runs against one
of the main applications of Segre classes: in principle one would want to compute Segre 
classes in order to solve hard enumerative problems, not the other way around. 
Formulas such as \eqref{eq:main} do not rely on enumerative information, so they 
have the potential for broader applications.

\subsection{}
We end this introduction by noting that the integrals appearing in \eqref{eq:main} have
the following property: if $T$ is an $n$-dimensional simplex, then
\begin{equation}\label{simplex}
\int_T \frac{n! X_1\cdots X_n\, da_1\cdots da_n}
{(1+a_1 X_1+\cdots +a_n X_n)^{n+1}}
=\frac{n! \Vol(T)\, X_1\cdots X_n}
{\prod_{(a_1,\dots,a_n)} (1+a_1 X_1+\cdots a_n X_n)}
\end{equation}
where the product ranges over the vertices of $T$. An analogous expression may be
given for unbounded regions dominating a simplex in lower dimension; see Proposition~\ref{calc}. 
Thus, one may
compute the integral in \eqref{eq:main} by splitting the region $N$ in any way into 
(possibly unbounded) simplices and applying~\eqref{simplex}. 

\begin{example}\label{introex3}
The computation carried out in Example~\ref{introex2} may be performed by splitting
the region $N$ as a union of triangles and one unbounded region as follows:
\begin{center}
\includegraphics[scale=.4]{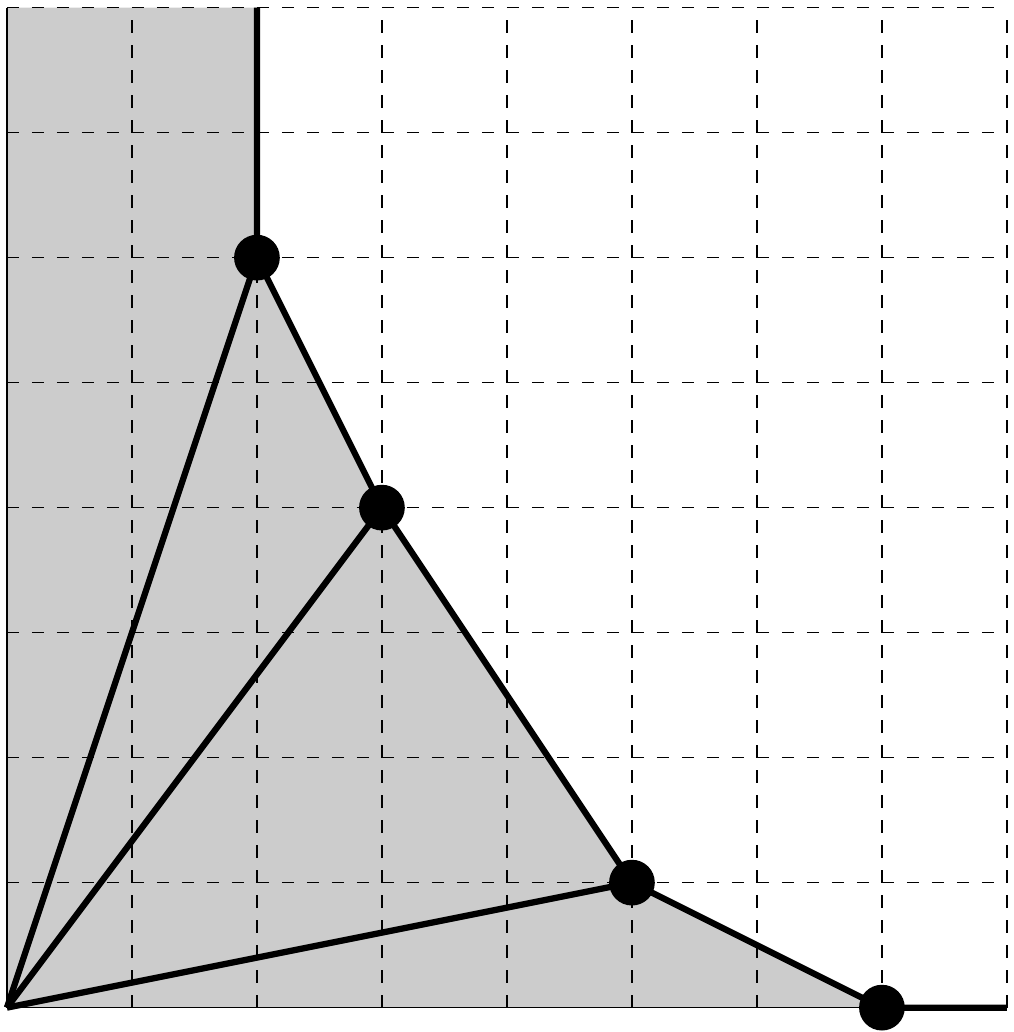}
\end{center}
Each triangle contribues a rational function according to \eqref{simplex}:
\begin{multline*}
\frac{2H}{(1+8H)}+\frac{10H^2}{(1+8H)(1+7H)}+\frac{17H^2}{(1+7H)(1+6H)}
+\frac{7H^2}{(1+6H)(1+7H)} \\
=\frac{2H(1+30H+168H^2)}{(1+6H)(1+7H)(1+8H)}
\end{multline*}
(the first term accounts for the unbounded region). This reproduces the result obtained
in Example~\ref{introex2}.
\qede\end{example}

This approach may in fact be taken as an alternative interpretation of the meaning of the 
right-hand side of~\eqref{eq:main}. What is possibly surprising from this point of view, and is
transparent from the interpretation as an integral, is that the result {\em does not depend 
on the chosen decomposition.\/}

Theorem~\ref{mainthm} is proven in \S\ref{proof}. Generalizations of \eqref{simplex} and 
examples giving evidence for the validity of \eqref{eq:main} for monomial ideals in 
arbitrarily many variables are discussed in~\S\ref{examples}.

\subsection{Acknowledgments}
The author's research is partially supported by a Simons collaboration grant.


\section{Proof of Theorem~\ref{mainthm}}\label{proof}

\subsection{Principal monomial ideals}\label{princip}
We maintain the notation used in the introduction. Notice that
\[
\int_{\Rbb_{\ge 0}^n} \frac{n!\,X_1\cdots X_n\, da_1\cdots da_n}
{(1+a_1 X_1+\cdots + a_n X_n)^{n+1}}
=1\quad;
\]
this follows from a simple induction. More generally, if $i_j\ge 0$ and $M$ 
denotes the region defined by the inequalities $a_1\ge i_1,\dots, a_n\ge i_n$, 
then
\[
\int_M \frac{n!\,X_1\cdots X_n\, da_1\cdots da_n}
{(1+a_1 X_1+\cdots + a_n X_n)^{n+1}}
=\frac 1{1+i_1 X_1+\cdots + i_n X_n}\quad.
\]
(Again, this is a simple induction; see Proposition~\ref{calc} for generalizations.)
\begin{center}
\includegraphics[scale=.4]{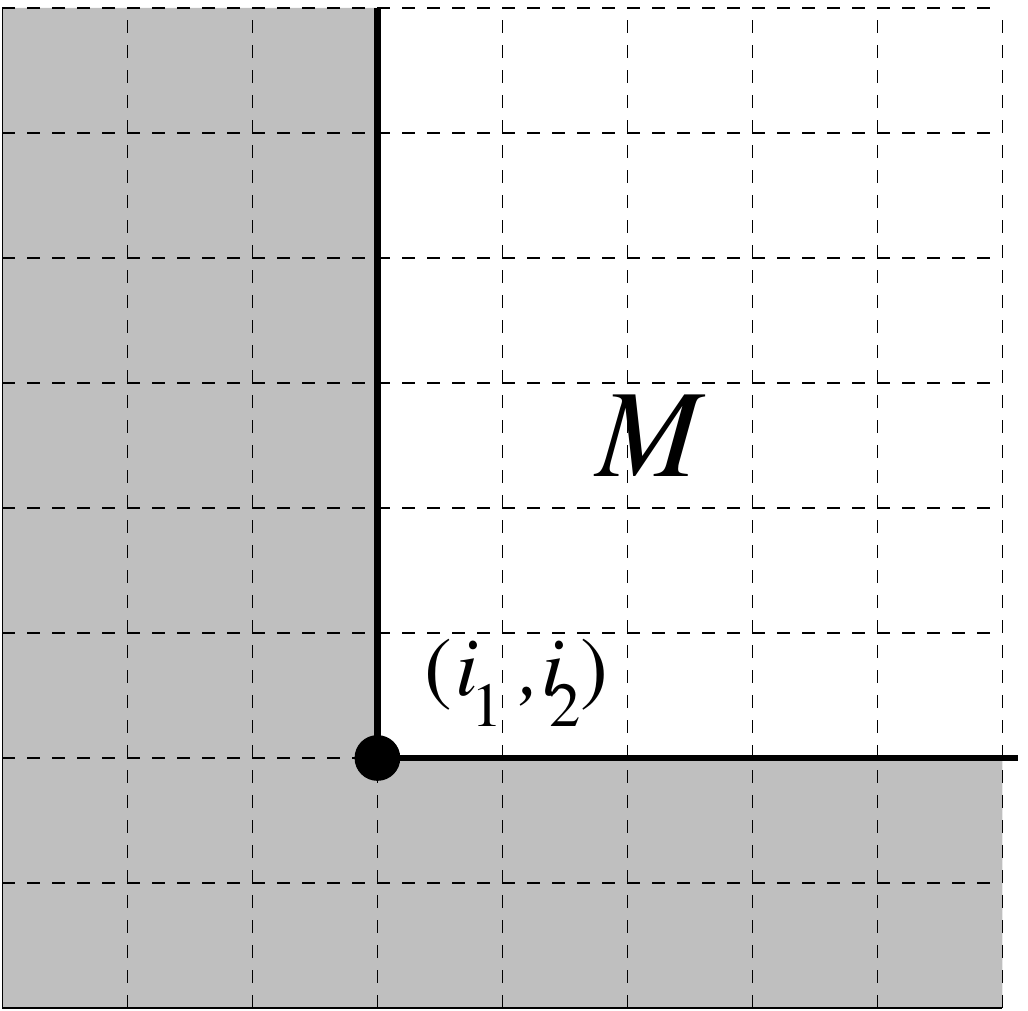}
\end{center}

Now consider the subscheme $S\overset \iota \hookrightarrow V$ defined by a 
principal ideal $(X^I)$ generated by a single monomial, with $I=(i_1,\dots,i_n)$. 
Then $S$ is a Cartier divisor, with normal bundle 
\[
N_SV \cong \cO(i_1 X_1 + \cdots + i_n X_n)\quad,
\]
and therefore
\[
\iota_* s(S,V) = c(N_SV)^{-1}\cap [S] = \frac{i_1 X_1+\cdots + i_n X_n}{1+i_1 X_1+\cdots
+i_n X_n}\quad.
\]
The integral from equation~\ref{eq:main} for this example is
\[
\int_{\Rbb_{\ge 0}^n\smallsetminus M} \frac{n!\, X_1\cdots X_n\, da_1\cdots da_n}
{(1+a_1 X_1+\cdots + a_n X_n)^{n+1}}=1-\frac 1{1+i_1 X_1+\cdots + i_n X_n}
=\frac {i_1 X_1+\cdots + i_n X_n}{1+i_1 X_1+\cdots + i_n X_n}
\]
verifying Conjecture~\ref{conj} for principal monomial ideals.

In particular, this verifies Conjecture~\ref{conj} for $n=1$.

\subsection{$n=2$}
Monomial ideals in two variables can be principalized by a sequence of blow ups along
codimension~2 loci. (In fact, {\em every\/} monomial ideal may be principalized by 
blowing up codimension~2 loci, cf.~\cite{MR2165388}.) A principalization algorithm
may be described as follows: if $S$ is a monomial ideal supported on $X_1, X_2$ in $V$, 
let $\pi: \Til V \to V$ be the blow-up of $V$ along $X_1\cap X_2$ (which is nonsingular
by hypothesis), let $E$ be the exceptional divisor, and $\Til X_1, \Til X_2$ the proper
transforms of $X_1$, $X_2$. (Note that $\Til X_1$, $\Til X_2$ are disjoint.)
Up to a principal component supported on $E$, $\pi^{-1}(S)$ 
is again a union of disjoint monomial subschemes $S_1, S_2$. Iterating this process on 
$S_1$, $S_2$ leads to a sequence of blow-ups principalizing $S$.

Therefore, the $n=2$ case of Conjecture~\ref{conj} may be proven by 
showing that the validity of the conjecture for $S_1$, $S_2$ implies that \eqref{eq:main}
holds for $S$: indeed, induction on the number of blow-ups needed for a principalization
reduces the conjecture to the principal case, which has been verified in \S\ref{princip}.
We carry out this strategy in the rest of this section.

Let $S$ be defined by monomials $X^{I_k}$, with $I_k=(i_{k1},i_{k2})$, $k=1,\dots, r$. The 
Newton region $N$ is the complement of the convex hull of the union of the $(i_{k1},
i_{k2})$ translations of the positive quadrants:
\begin{center}
\includegraphics[scale=.4]{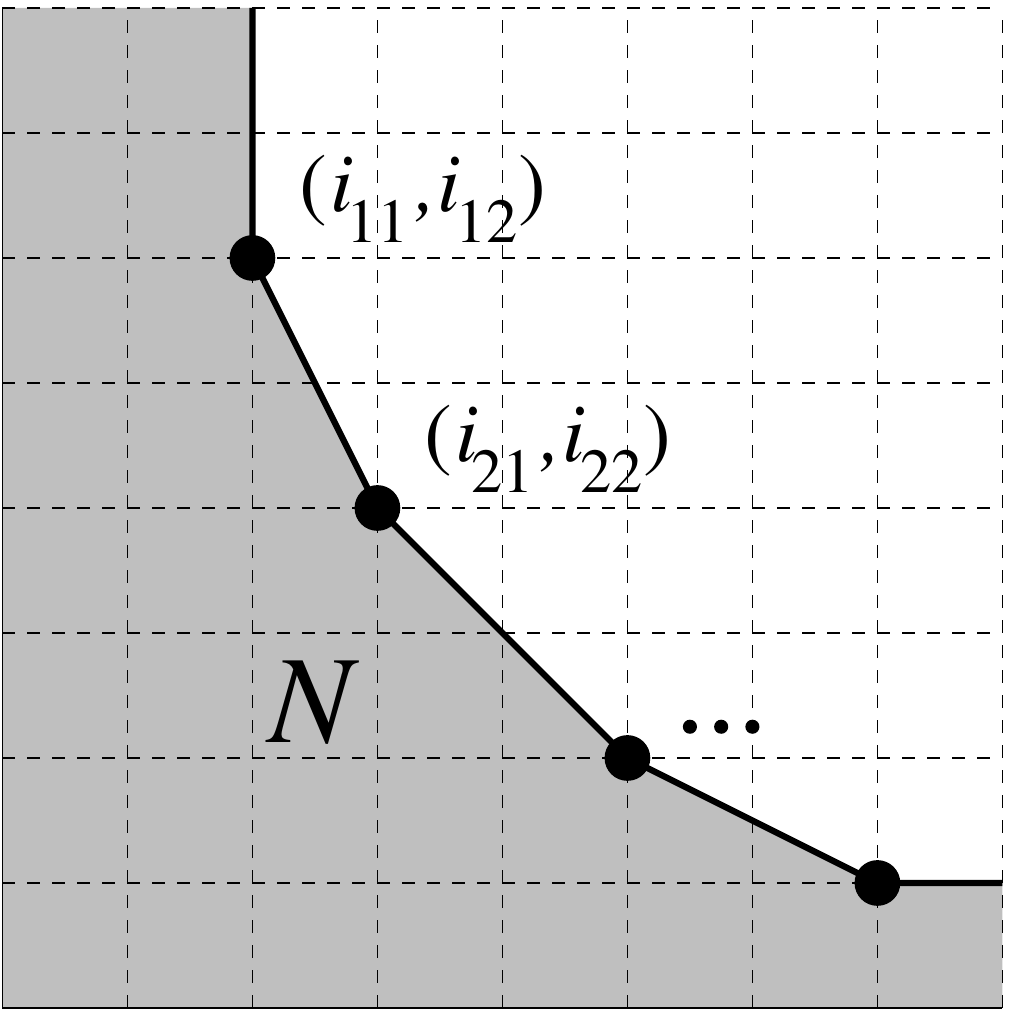}
\end{center}
For $n=2$, \eqref{eq:main} states that
\[
\iota_* s(S,V) = \int_N \frac{2\, X_1 X_2\, da_1 da_2}{(1+a_1 X_1 + a_2 X_2)^3}\quad,
\]
where $\iota: S\hookrightarrow V$ is the embedding. With $\pi: \Til V\to V$ as above,
let $j$ be the inclusion of $\pi^{-1}(S)$ in $\Til V$. 
\[
\xymatrix{
\pi^{-1}(S) \ar@{^(->}[r]^-j \ar[d] & \Til V \ar[d]^\pi \\
S \ar@{^(->}[r]^\iota & V
}
\]
By the birational invariance of Segre classes (Proposition~4.2 (a) in \cite{85k:14004}),
\[
\iota_* s(S,V) = \pi_* j_* s(\pi^{-1}(S),\Til V)\quad.
\]
The scheme $\pi^{-1}(S)$ contains a copy of the exceptional divisor $E$ with multiplicity 
$m$, where $m$ is the minimum of $i_{k1}+i_{k2}$ for $k=1,\dots, r$. 
The region $N$ is split into three areas: the triangle $T$ with vertices $(0,0),
(0,m),(m,0)$, and the two (possibly empty) top-left ($N'$) and bottom-right ($N''$) components 
of the complement of this triangle:
\begin{center}
\includegraphics[scale=.4]{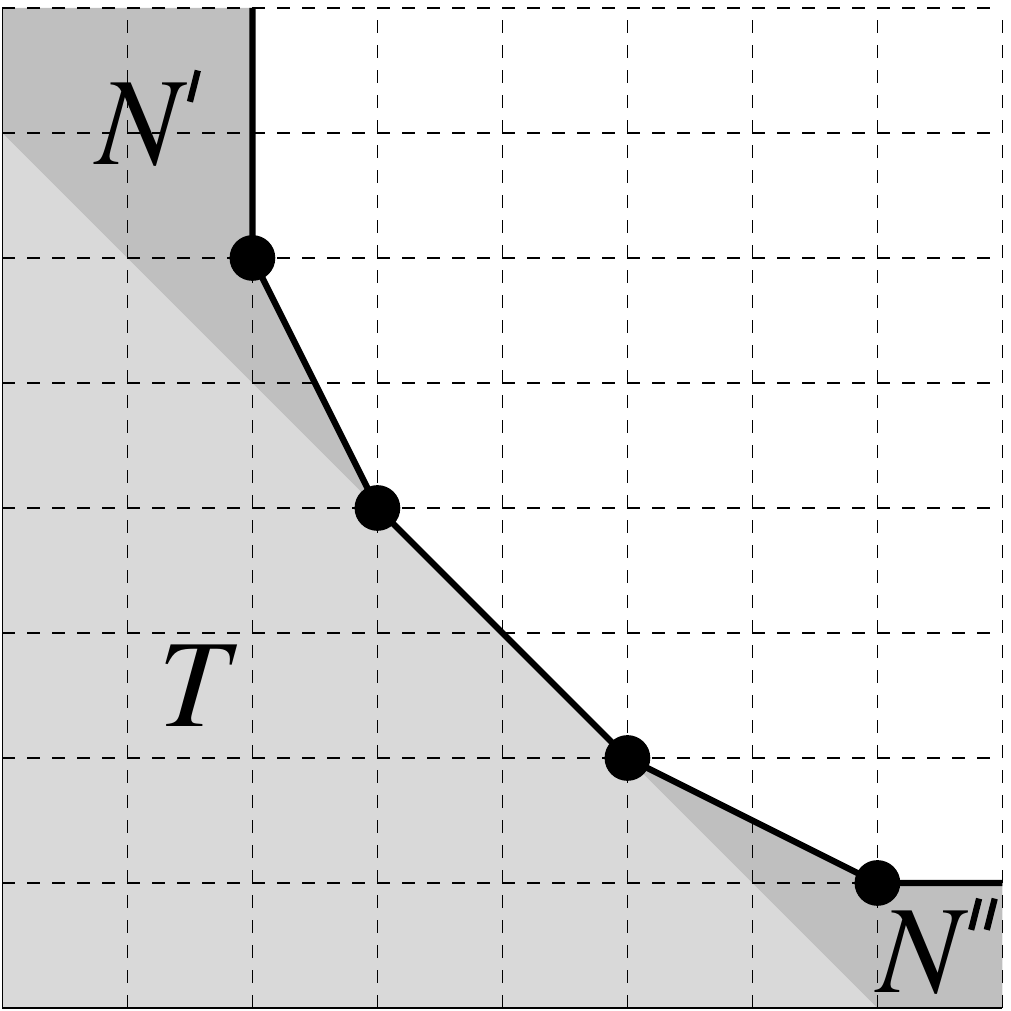}
\end{center}

The residual of $mE$ in $\pi^{-1}(S)$ consists of two disjoint subschemes $S_1$, $S_2$ 
of $\Til V$. These are monomial subschemes, supported respectively on $\Til X_1, E$ 
and $E, \Til X_2$. 

\begin{claim}\label{expo}
Exponents for $S_1$, resp.~$S_2$ are
\[
(i_{k1},i_{k1}+i_{k2}-m), \quad k=1,\dots,r,
\quad\text{resp.}\quad
(i_{k1}+i_{k2}-m,i_{k2}), \quad k=1,\dots,r\quad.
\]
\end{claim}

\begin{proof}
In a neighborhood of $S_1$, $\pi^{-1}(S)$ is the intersection of divisors denoted additively as
\[
i_{k1} \pi^{-1}(X_1) + i_{k2} \pi^{-1}(X_2) = i_{k1} (\Til X_1 +E) + i_{k2} E
=i_{k1}(\Til X_1) + (i_{k_1}+i_{k_2}) E
\]
(since $S_1$ is disjoint from $\Til X_2$, $\pi^{-1}(X_2)$ agrees with $E$ near $S_1$). 
The monomial scheme $S_1$ is obtained as the residual of $mE$ in this intersection, 
with corresponding exponents as stated. The analysis is identical near $S_2$.
\end{proof}

By residual intersection (\cite{85k:14004}, Proposition~9.2, in the form given 
in~\cite{MR96d:14004}, Proposition~3), $s(\pi^{-1}(S),\Til V)$ equals
\[
\frac{mE}{1+mE} + \frac 1{1+mE}\cap (s(S_1,\Til V)\otimes \cO(mE))
+\frac 1{1+mE}\cap (s(S_2,\Til V)\otimes \cO(mE))\quad.
\]
(Here and in what follows we use freely notation as in~\cite{MR96d:14004}, \S2.)
Therefore, $i_* s(S,V)$ is naturally the sum of three terms: the push-forwards to $V$ of
\[
(\text{i})\, \frac{mE}{1+mE}, \quad
(\text{ii})\, \frac 1{1+mE}\cap (s(S_1,\Til V)\otimes \cO(mE)), \quad
(\text{iii})\, \frac 1{1+mE}\cap (s(S_2,\Til V)\otimes \cO(mE))\quad.
\]
The following claim will conclude the proof of Theorem~\ref{mainthm}.

\begin{claim}\label{keyclaim}
The terms (i), resp.~(ii), (iii) push-forward to the values of the integral on the subregions 
$T$, resp.~$N'$, $N''$ of $N$ determined above.
\end{claim}

\begin{proof}
---(i): By the birational invariance of Segre classes, the push-forward of $E/(1+E)$
is the Segre class of the center of the blow-up. Therefore, the push-forward of $mE/(1+mE)$ 
is the $m$-th Adams of this Segre class. Since the center is the complete intersection of
$X_1$ and $X_2$, this is given by
\[
\frac {mE}{1+mE} \mapsto \frac {m^2 X_1\cdot X_2}{(1+mX_1)(1+mX_2)}\quad.
\]
(The Segre class of a complete intersection equals the inverse Chern class of its normal 
bundle.) The claim is that this expression equals the integral
\[
\int_T \frac{2 X_1X_2\, da_1\, da_2}{(1+a_1 X_1+a_2 X_2)^3} 
\]
where $T$ is the triangle with vertices $(0,0)$, $(0,m)$, $(m,0)$. The verification of
this fact is a trivial calculus exercise. (See Proposition~\ref{calc} for a generalization.)
\smallskip

---(ii): Term (ii) is
\[
\frac 1{1+mE}\cap (s(S_1,\Til V)\otimes \cO(mE))\quad,
\]
where $S_1$ is monomial on $\Til X_1$, $E$ with exponents $(i_{k1},i_{k1}+i_{k2}-m)$ as
observed in Claim~\ref{expo}. Each vertex of the Newton polyhedron
for $S$ determines a vertex for $S_1$ by the transformation $(a_1,a_2)\mapsto
(\tilde a,e)=(a_1, a_1+a_2-m)$.
\begin{center}
\includegraphics[scale=.35]{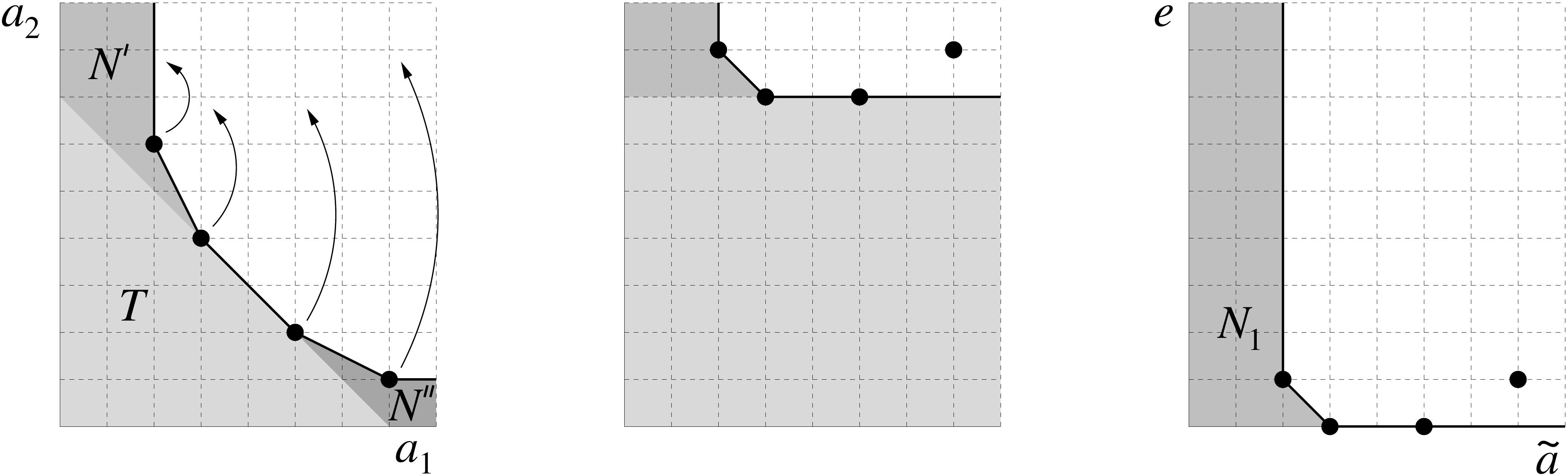}
\end{center}
By induction on the number of blow-ups needed for a principalization, Conjecture~\ref{conj} 
holds for $S_1$. Thus
\begin{equation}\label{indu}
\iota_{1*} s(S_1,\Til V)
=\int_{N_1} \frac{2 \Til X_1E \, d\tilde a\, de}{(1+\tilde a \Til X_1+eE)^3}
\end{equation}
where $\iota_1: S_1\hookrightarrow \Til V$ is the embedding, and $N_1$ is the Newton
region for $S_1$. Note that $N_1$ maps onto region $N'$ via the transformation
$(\tilde a,e)\mapsto (a_1,a_2)=(\tilde a,e-\tilde a+m)$.

\begin{claim}
\[
\iota_{1*}\left(\frac 1{1+mE}\cap (s(S_1,\Til V)\otimes \cO(mE))\right)
=\int_{N_1} \frac{2 \Til X_1E\, d\tilde a\, de}{(1+\tilde a \Til X_1+ (e+m)E)^3} 
\]
\end{claim}

\begin{proof}
Applying \eqref{indu}, the projection formula, and the fact that the line bundle $\cO(mE)$ 
is constant with respect to the integration variables $\tilde a$ and $e$, we see that the 
left-hand side equals
\begin{align*}
\frac 1{1+mE} \cap &
\left(\left(\int_{N_1} \frac{2 \Til X_1E \, d\tilde a\, de}{(1+\tilde a \Til X_1+eE)^3}\right)
\otimes \cO(mE)\right) \\
&=\int_{N_1}\frac 1{1+mE}  \left(
\frac{2 \Til X_1E}{(1+\tilde a \Til X_1+eE)^3} 
\otimes \cO(mE)\right)\, d\tilde a\, de \\
&=\int_{N_1}\frac 1{1+mE}  \left(
\frac{(1+mE)\, 2 \Til X_1E}{(1+\tilde a \Til X_1+(e+m)E)^3} 
\right)\, d\tilde a\, de \\
&=\int_{N_1}
\frac{2 \Til X_1E\, d\tilde a\, de}{(1+\tilde a \Til X_1+(e+m)E)^3}  
\end{align*}
as stated. (We have used here the formal properties of the $\otimes$ operation,
in particular Proposition~1 of \cite{MR96d:14004}.)
\end{proof}

Thus, we are reduced to proving

\begin{claim}
\begin{equation}\label{subst}
\pi_* \int_{N_1} \frac{2 \Til X_1E\, d\tilde a\, de}{(1+\tilde a \Til X_1+ (e+m)E)^3} 
=\int_{N'} \frac{2\, X_1 X_2\, da_1 da_2}{(1+a_1 X_1 + a_2 X_2)^3}\quad.
\end{equation}
\end{claim}

\begin{proof}
As observed above, $N_1$ maps onto $N'$ via $(a_1,a_2)\mapsto (\tilde a,e)=
(a_1, a_1+a_2-m)$. This transformation has Jacobian 1, therefore \eqref{subst}
follows from the equality
\begin{equation}\label{pf}
\pi_*\left(\frac{\Til X_1\cdot E}{(1+\tilde a \Til X_1+ (e+m)E)^3}\right)
=\frac{X_1\cdot X_2}{(1+a_1 X_1 + a_2 X_2)^3}\quad.
\end{equation}
In turn, \eqref{pf} follows from the projection formula. Indeed, 
\begin{align*}
\pi^*(a_1 X_1+a_2 X_2) &= a_1 (\Til X_1+E) + a_2(\Til X_2+E)
=\tilde a(\Til X_1+E) + (e-\tilde a+m)(\Til X_2+E)
\\ &=\tilde a \Til X_1 + (e+m) E +(e-\tilde a+m)\Til X_2\quad,
\end{align*}
so that $\dfrac{\Til X_1\cdot E}{(1+\tilde a\Til X_1+(e+m)E)^3}$ equals
\[
\frac{\Til X_1\cdot E}{(\pi^*(1+a_1 X_1+a_2 X_2)-(e-\tilde a+m)\Til X_2)^3}
=\frac{\Til X_1\cdot E}{(\pi^*(1+a_1 X_1+a_2 X_2))^3}
\]
as $\Til X_1\cdot \Til X_2=0$. Hence
\[
\pi_*\left(\frac{\Til X_1\cdot E}{(1+\tilde a \Til X_1+ (e+m)E)^3}\right)
=\pi_*\left(\frac{\Til X_1\cdot E}{(\pi^*(1+a_1 X_1+a_2 X_2))^3}\right)
=\frac{\pi_*(\Til X_1\cdot E)}{(1+a_1 X_1 + a_2 X_2)^3}\quad,
\]
and $\pi_* (\Til X_1\cdot E)=X_1\cdot X_2$, concluding the proof.
\end{proof}

---(iii) is handled in exactly the same way.
\end{proof}

\begin{remark}\label{intdep}
As a consequence of Theorem~\ref{mainthm}, monomial generators which do
not affect the Newton region (i.e., which are in the convex hull of the translated
quadrants determined by the other generators) do not affect the Segre class.

This fact is not surprising, and holds for arbitrary $n$. Indeed, such generators
do not affect the integral closure of the ideal, and the Segre class only depends on
the integral closure, cf.~the proof of Lemma~1.2 in~\cite{MR97b:14057}.
\qede\end{remark}


\section{Examples for arbitrary $n$}\label{examples}

\subsection{Calculus}\label{calculus}
The following observation simplifies the computation of the integral in Conjecture~\ref{conj}.
We work in $\Rbb^n$, with coordinates $(a_1,\dots,a_n)$.
Let $\ue_1=(1,0,\dots,0),\dots, \ue_n=(0,\dots,0,1)$. We also consider indeterminates 
$X_i$, and for $\ua=(a_1,\dots,a_n)\in \Rbb^n$ we write $a_1 X_1+\cdots
+a_n X_n= \ua\cdot \uX$. The integral in \eqref{eq:main} is
\[
\int_N \frac{n!\, X_1\cdots X_n\, da_1\cdots da_n}{(1+\ua\cdot \uX)^{n+1}}\quad.
\]

Let $T$ be a $k$-dimensional simplex in $\Rbb^n$, with vertices $\uv_0,\dots, \uv_k$. 
For $J=\{j_1,\dots,j_k\}$, with $1\le j_i<\cdots< j_k\le n$, let $\pi_J: \Rbb^n \to \Rbb^k$
be the projection to the span of $\ue_{j_1},\dots,\ue_{j_k}$. We denote by $T^J$ the 
region $\{\ua+\sum_{i\not\in J} \lambda_i \ue_i\,|\, \ua\in T, \lambda_i\ge 0\}$.

\begin{center}
\includegraphics[scale=.5]{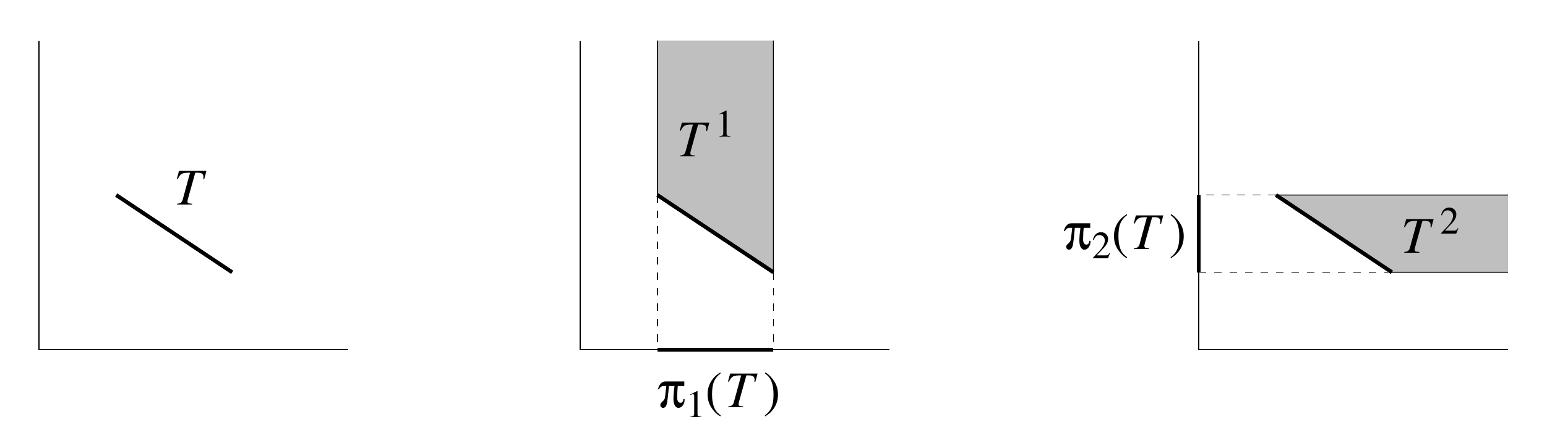}
\end{center}

\begin{prop}\label{calc}
With notation as above,
\[
\int_{T^J} \frac{n!\, X_1\cdots X_n\, da_1\cdots da_n}{(1+\ua\cdot \uX)^{n+1}}
=\frac{k!\Vol(\pi_J(T))\,X_{j_1}\cdots X_{j_k}}{(1+\uv_0\cdot \uX)\cdots (1+\uv_k\cdot \uX)}\quad.
\]
\end{prop}

\begin{proof}
For simplicity of notation, assume $J=\{1,\dots,k\}$.
If $\dim \pi_J(T)<k=\dim T$, then the integral is $0$. Thus, we may assume that $\pi_J(T)$ is
a $k$-simplex, with vertices $\uv_i'=\pi_J(\uv_i)$. Parametrize $\pi_J(T)$ by the simplex $\Sigma$ 
defined by 
\[
\{(t_1,\dots, t_k)\,|\, t_i\ge 0, t_1+\cdots +t_k\le 1\}
\]
in $\Rbb_{\ge 0}^k$, via
\[
(t_1,\dots, t_k) \mapsto \uv'(\ut):=\uv'_0+t_1(\uv'_1-\uv'_0)+\cdots +t_k (\uv'_k-\uv'_0)\quad.
\]
The Jacobian of this parametrization is $\Vol(\pi_J(T))$. There are linear functions
$\ell_j(\ut)$ such that
\[
(t_1,\dots, t_k) \mapsto \uv(\ut):=\uv'(\ut)+\sum_{j=k+1}^n \ell_j(\ut)\, \ue_j
\]
parametrizes $T$. With this notation,
\[
\uv(0,\dots,0)=\uv_0,\quad
\uv(1,0,\dots,0)=\uv_1,\quad
\cdots,\quad
\uv(0,\dots,0,1)=\uv_k,
\]
and the integral to be evaluated is
\[
\int_\Sigma \int_{\ell_{k+1}(\ut)}^{\infty} \cdots \int_{\ell_n(\ut)}^{\infty}
\frac{\Vol(\pi_J(T))\, X_1\cdots X_n\, n!\,da_n\cdots da_{k+1}\, dt_k\cdots dt_1}
{(1+\uv'(\ut)\cdot (X_1,\dots,X_k)+a_{k+1} X_{k+1}+\cdots+a_n X_n)^{n+1}} 
\quad.
\]
Performing the unbound integrations shows that this equals
\begin{multline*}
\int_\Sigma\frac{\Vol(\pi_J(T))\,X_1\cdots X_k \,k!\, dt_k\cdots dt_1}
{(1 + \uv'(\ut)\cdot (X_1,\dots,X_k)+\ell_{k+1}(\ut) X_{k+1}+\cdots +\ell_n(\ut) X_n)^{n+1}}\\
=\Vol(\pi_J(T))\,X_1\cdots X_k \,\int_\Sigma\frac{k!\, dt_k\cdots dt_1}
{(1 + \uv(\ut)\cdot \uX)^{k+1}}
\end{multline*}
Thus, the statement of Proposition~\ref{calc} is reduced to the following fact:
\[
\int_\Sigma \frac{k!\, dt_k\cdots dt_1}{(M+L_1 t_1+\cdots +L_kt_k)^{k+1}}
=\frac 1{M(M+L_1)\cdots (M+L_k)}\quad,
\]
with $M, L_1,\dots, L_k$ independent of $\ut$. This equality is immediately verified 
by induction.
\end{proof}

\begin{example}\label{exas}
---The integral over the unbounded region in Example~\ref{introex3} is evaluated by
Proposition~\ref{calc} with $T=$ the segment $(0,0) \to (2,6)$ and $J=\{1\}$:
\begin{center}
\includegraphics[scale=.3]{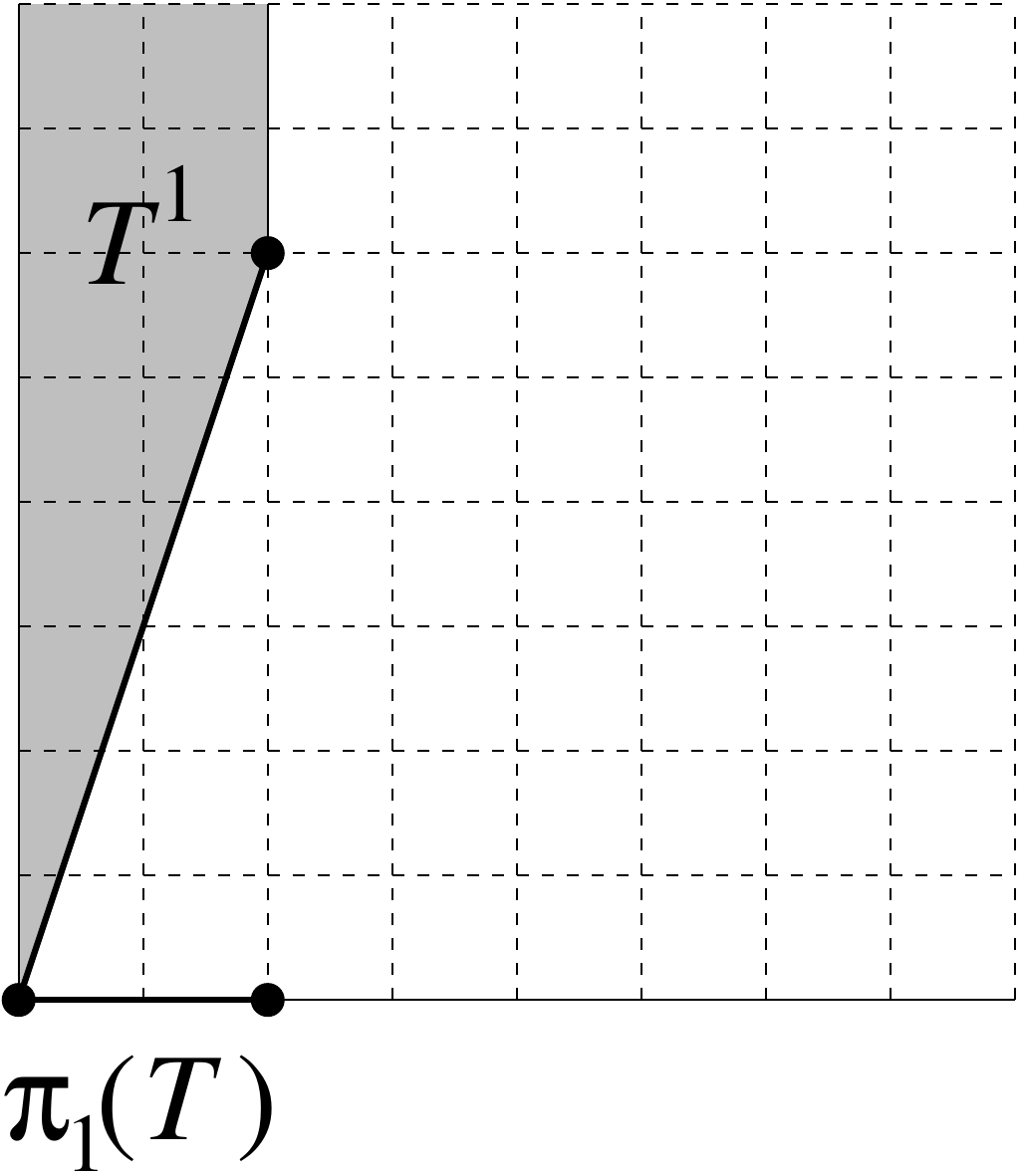}
\end{center}
\[
\int_{T^1} \frac{2!\, X_1 X_2\, da_1 da_2}{(1+\ua\cdot \uX)^3}
=\frac{1!\,2\,X_1}{(1+0X_1+0X_2)(1+2X_1+6X_2)}\quad.
\]
With $X_1=X_2=H$, this gives the term $\dfrac{2H}{1+8H}$ used in Example~\ref{introex3}.
\smallskip

---The integral over the triangle $T$ with vertices $(0,0)$, $(0,m)$, $(m,0)$ is evaluated
by taking $J=\{1,2\}$, giving the expression used in the proof of Claim~\ref{keyclaim} (i).
\smallskip

---For a more interesting example, consider the `infinite column' $T^{12}$ determined by 
the triangle in $\Rbb^3$ with vertices $(0,0,0)$, $(1,0,1)$, $(0,1,1)$ by taking $J=\{1,2\}$:
\begin{center}
\includegraphics[scale=.5]{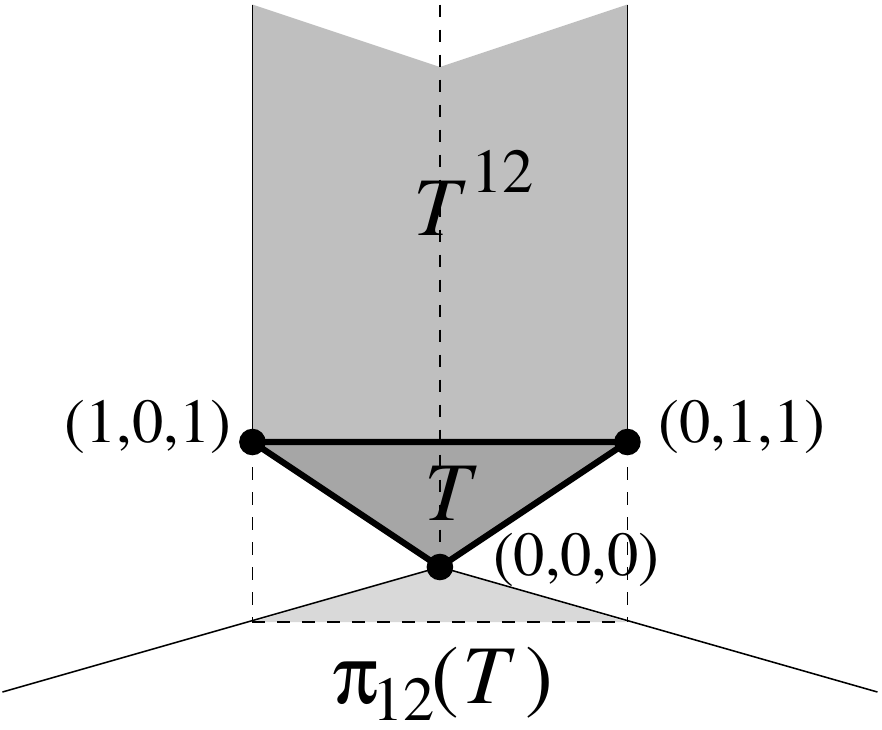}
\end{center}

The area of $\pi_{12}(T)$ is $\frac 12$, so Proposition~\ref{calc} evaluates the integral over 
$T^{12}$ as
\[
\frac{X_1\cdot X_2}{(1+X_1+X_3)(1+X_2+X_3)}\quad.
\]
This will be used below in Example~\ref{threelines}.
\qede\end{example}

\subsection{Complete intersections}\label{ci}
As we are assuming that $X_1,\dots,X_n$ are nonsingular hypersurfaces meeting
with normal crossings, the monomial subscheme corresponding to exponents
\[
(m_1,0,\dots,0),\quad
(0,m_2,0,\dots,0),\quad
\cdots,\quad
(0,\dots,0,m_n)
\]
is a complete intersection, with normal bundle $\cO(m_1 X_1)\oplus\cdots\oplus \cO(m_n
X_n)$. Therefore
\[
\iota_* s(S,V) = \frac{[S]}{\prod_i c(\cO(m_i X_i)} = \prod_i \frac{m_i X_i}{(1+m_i X_i)}\quad.
\]
The corresponding Newton region is a simplex with vertices on the coordinate axes:
\begin{center}
\includegraphics[scale=.5]{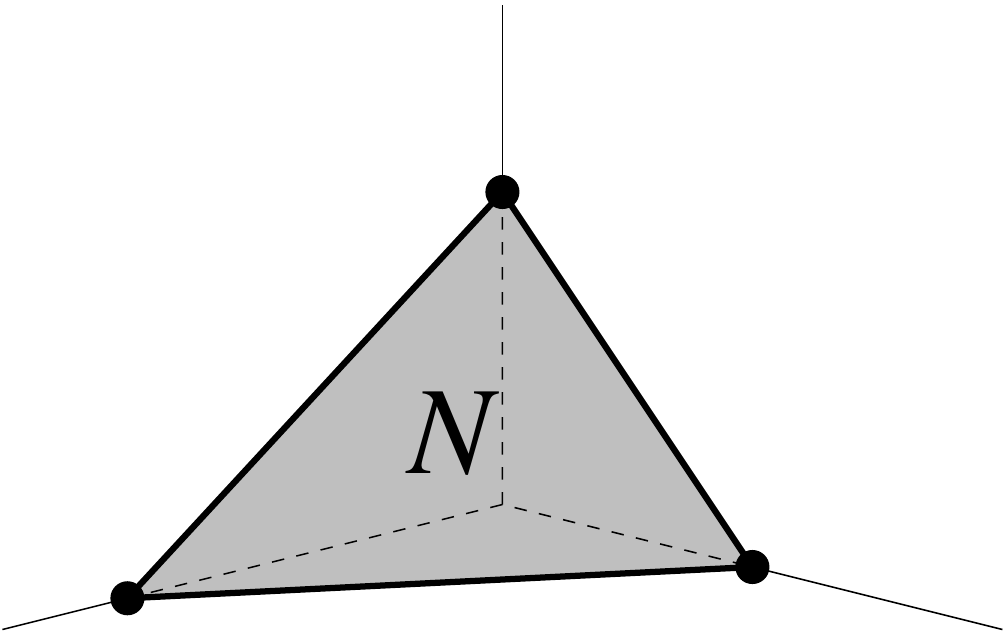}
\end{center}
and according to Proposition~\ref{calc}
\begin{align*}
\int_N \frac{n! X_1\cdots X_n\, da_1\cdots da_n}{(1+a_1 X_1+\cdots+a_n X_n)^{n+1}}
&=\frac{n! \Vol(N)\, X_1\cdots X_n}{(1+m_1X_1)\cdots (1+m_n X_n)} \\
&=\frac{m_1\cdots m_n\, X_1\cdots X_n}{(1+m_1X_1)\cdots (1+m_n X_n)}\quad.
\end{align*}
This verifies Conjecture~\ref{conj} for these complete intersection.\smallskip

A somewhat harder calculus exercise verifies Conjecture~\ref{conj} for arbitrary 
complete intersections of monomials. As an example of what is involved in this
verification, consider a monomial subscheme $S'\overset{\iota'} \hookrightarrow V$ 
supported on $X_1,\dots,X_{n-1}$, and let $S$ be the intersection of~$S'$ with the 
$m$-multiple $m X_n$. Standard facts about Segre classes imply that
\begin{equation}\label{eq:ci}
\iota'_* s(S',V) = \frac{mX_n}{1+mX_n}\cap \iota_* s(S,V)\quad.
\end{equation}
To see that Conjecture~\ref{conj} is compatible with this formula, observe that the Newton
region $N$ for $S$ is a cone over the Newton region $N'$ for $S'$ with vertex at
$(0,\cdots,0,m)$.
\begin{center}
\includegraphics[scale=.5]{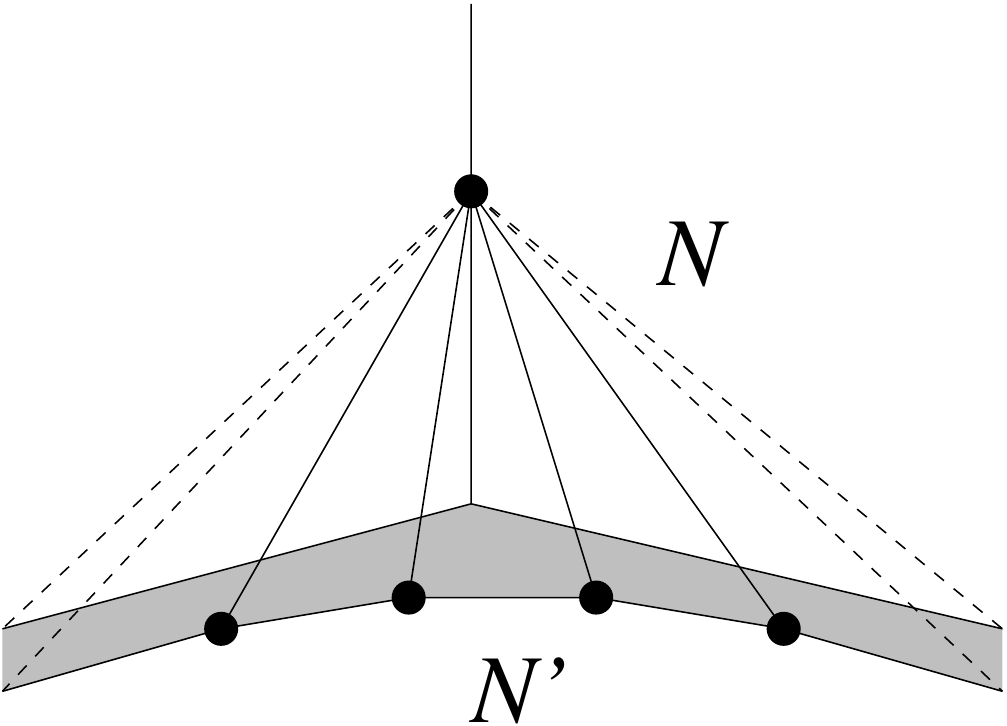}
\end{center}
The equality \eqref{eq:ci} amounts to
\[
\int_N\frac{n!\, X_1\cdots X_n\, da_1\cdots da_n}{(1+a_1 X_1+\cdots+a_n X_n)^{n+1}}
=\frac{mX_n}{1+mX_n}\cap \int_{N'} 
\frac{(n-1)!\, X_1\cdots X_{n-1}\, da_1\cdots da_{n-1}}{(1+a_1 X_1+\cdots+a_{n-1} X_{n-1})^n}
\]
and in the cone situation this follows from
\[
\int_0^1 \frac{n!\,t^{n-1} dt}{(A+Bt)^{n+1}} = \frac{(n-1)!}{A(A+B)^n}
\]
with suitable positions for $A$ and $B$, as the reader may check. A similar (but harder) computation
verifies the corresponding formula whenever $S'$ is a monomial scheme in $X_1,\dots, X_k$ 
and the lone vertex is a single monomial in $X_{k+1},\dots, X_n$. Repeated application of this 
more general formula implies that \eqref{eq:main} holds for any complete intersection of 
monomials.

\subsection{Singularity subschemes}
As a less straightforward family of examples verifying Conjecture~\ref{conj}, we consider
the monomial subschemes on $X_1,\dots,X_n$ with exponents
\begin{equation}\label{eq:expos}
\uf_1:=(0,1,\dots, 1),\quad \uf_2:=(1,0,1,\dots,1),\quad \cdots,\quad \uf_n:=(1,\dots,1,0)\quad.
\end{equation}
These subschemes are very far from being complete intersections (for $n>1$), and 
computing their Segre class requires some nontrivial work, which depends on features 
of these schemes which do not hold for arbitrary monomial schemes. {\em Ad-hoc\/} 
alternative methods are occasionally available, as in the following example.

\begin{example}\label{threelines}
For $n=3$ in $\Pbb^3$, with coordinates $x,y,z,w$, the exponents $\uf_1,\uf_2,\uf_3$
define the monomial subscheme $S$ with ideal
\[
(xy, xz,yz)\quad.
\]
This scheme is reduced and consists of three coordinate axes, so its Segre class must
be $\iota_* s(S,\Pbb^3)=3[\Pbb^1]+a [\Pbb^0]$ for some integer $a$. On the other hand,
$S$ is the scheme-theoretic intersection of three quadrics $Q_1$, $Q_2$, $Q_3$ (each 
consisting of two coordinate planes). By Fulton-MacPherson intersection theory and 
B\'ezout's theorem, and denoting by $H$ the hyperplane class,
\[
8 = \int Q_1\cdot Q_2 \cdot Q_3 = \int c(\cO(2))\cap s(S,\Pbb^3)
=\int (1+2H)^3\cap (3[\Pbb^1]+a [\Pbb^0])
=18+a\quad.
\]
It follows that $a=-10$, so $\iota_* s(S,\Pbb^3)=3[\Pbb^1]-10 [\Pbb^0]$.

The Newton region for $S$ may be decomposed into $3$ infinite columns  and
one tetrahedron with vertices $(0,0,0),(1,1,0),(1,0,1),(0,1,1)$.
\begin{center}
\includegraphics[scale=.5]{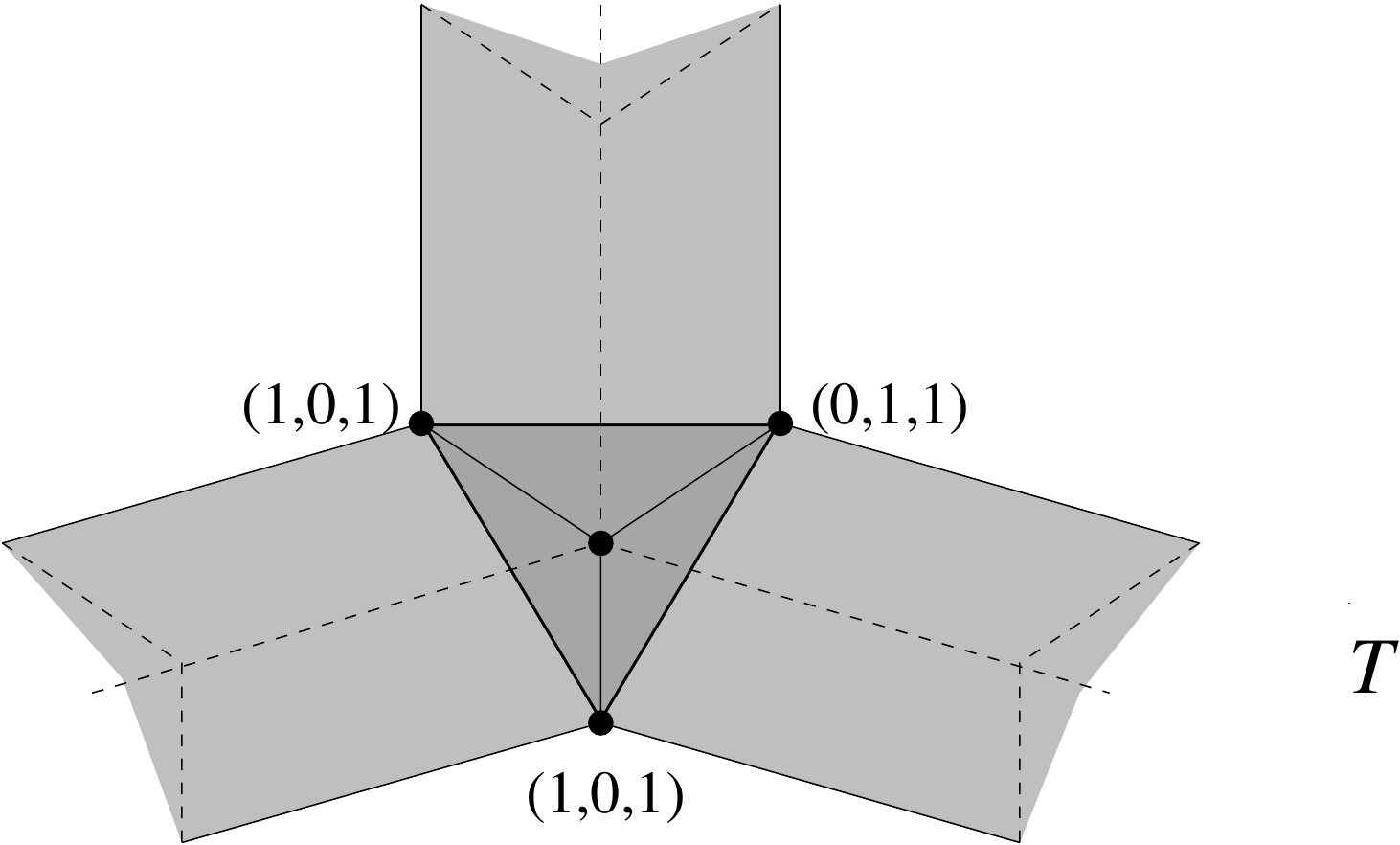}
\end{center}
The contribution of a column to the integral in \eqref{eq:main} was computed in
Example~\ref{exas}, and equals (setting $X_1=X_2=H$, the hyperplane class)
\[
\frac{H^2}{(1+2H)^2}\quad.
\]
The tetrahedron has volume $\frac 13$, hence it contributes (using 
Proposition~\ref{calc}, and again setting $X_1=X_2=X_3=H$)
\[
\frac{3!\, \dfrac 13\, X_1 X_2 X_3}{(1+X_1+X_2)(1+X_1+X_3)(1+X_2+X_3)}
=\frac{2H^3}{(1+2H)^3}\quad.
\]
Therefore, according to Conjecture~\ref{conj}, the Segre class equals
\[
\int_N \frac{n!\,X_1 X_2 X_3\,da_1 da_2 da_3}{(1+a_1 X_1+a_2 X_2+a_3 X_3)^4}
=\frac{2H^3}{(1+2H)^3}+3 \frac{H^2}{(1+2H)^2}=\frac{3H^2+8H^3}{(1+2H)^3}
=3H^2-10 H^3
\]
in agreement with the direct computation.
\qede\end{example}

For any $n$, the subscheme $S$ defined by the exponents~\eqref{eq:expos} is the
{\em singularity subscheme\/} of the union $X$ of the hypersurfaces $X_1,\dots, X_n$,
i.e., the subscheme of $X$ locally defined by the partials of an equation for $X$ in $V$.
(For example, $(xy,xz,yz)$ is the ideal generated by the partials of $xyz$.)
This is what gives us independent access to the Segre classes for these subschemes, 
and allows to verify Conjecture~\ref{conj} in these cases.

\begin{prop}\label{nontri}
For all $n>0$, Conjecture~\ref{conj} holds for the monomial subschemes defined 
by the exponents $\uf_1,\dots,\uf_n$ listed in \eqref{eq:expos}.
\end{prop}

\begin{proof}
Let $F=\{\uf_1,\dots,\uf_n\}$ be the set of exponents. The Newton region $N$ may be
described as follows. For any $J\subseteq \{1,\dots,n\}$, let $\Sigma_J$ 
be the simplex with vertices at the origin and at the points $\uf_j$ with $j\in J$, and 
consider the subsets $\Sigma_J^J$, with notation as in~\S\ref{calculus}.
The reader can verify that the Newton region $N$ is then the union of the sets 
$\Sigma_J^J$ with $|J|\ge 2$.
(For example, the region $N$ for $n=3$ decomposes as the union of the three 
columns $\Sigma_{12}^{12}$, $\Sigma_{13}^{13}$, $\Sigma_{23}^{23}$, and the 
$3$-simplex $\Sigma_{123}^{123}$. Cf.~Example~\ref{threelines}.)

The volume of the projection $\pi_J(\Sigma_J)$ is easily found to be $(|J|-1)/|J|!$. 
By Proposition~\ref{calc},
\[
\int_N \frac{n!\,X_1\cdots X_n\, da_1\cdots da_n}{(1+a_1X_1+\cdots+a_nX_n)^{n+1}}
=\sum_{|J|\ge 2} \frac{(|J|-1) \prod_{j\in J} X_j}{\prod_{j\in J} (1+\uf_j\cdot \uX)}\quad.
\]
We have to verify that this equals the class $\iota_* s(S,V)$. This Segre class is
computed in~\cite{MR2001i:14009}, \S2.2 (p.~4002):
\[
\iota_* s(S,V)=\left(\left(
1-\frac{c(\cL^\vee)}{c(\cL_1^\vee)\cdots c(\cL_n^\vee)}\right)\cap [V]\right)
\otimes \cL
=\left(1-\frac{c(\cL)^{n-1}}{c(\cL\otimes \cL_1)\cdots c(\cL\otimes \cL_n)}\right)\cap [V]
\]
where $\cL_i=\cO(X_i)$, $\cL=\cO(X_1+\cdots +X_n)$. (This is an instance of the
relation between the Chern-Schwartz-MacPherson class of a hypersurface and the
Segre class of its singularity subscheme, in the particular case of divisors with normal
crossings.) 
Thus, we have to verify that
\begin{equation}\label{eq:unlike}
\sum_{|J|\ge 2} \frac{(|J|-1) \prod_{j\in J} X_j}{\prod_{j\in J} (1+\uf_j\cdot \uX)}
=1-\frac{(1+X_1+\cdots+X_n)^{n-1}}{(1+\uf_1\cdot \uX)\cdots (1+\uf_n\cdot \uX)}\quad.
\end{equation}
This is in fact an identity of rational functions in indeterminates $X_1,\dots,X_n$. 
To prove it, interpret the left-hand side as the value at $t=1$ of
\begin{multline*}
\frac d{dt} \left(\frac 1t \sum_{|J|\ge 2} \prod_{j\in J} \frac{X_j\, t}{1+ \uf_j\cdot \uX}
\right)
=\frac d{dt} \left(\frac 1t \left(\prod_{j=1}^n \left(1+\frac {X_j\, t}{1+\uf_j\cdot \uX}
\right)-1\right)\right)\\
=\frac 1{\prod_j (1+\uf_j\cdot \uX)}\cdot \frac d{dt} \left(\frac 1t \prod_j(1+\uf_j\cdot \uX
+X_j\,t)-\frac 1t \prod_j(1+\uf_j\cdot X)\right)
\end{multline*}
As
\begin{multline*}
\frac d{dt} \left(\frac 1t \prod_j(1+\uf_j\cdot \uX+X_j\,t)-\frac 1t \prod_j(1+\uf_j\cdot X)\right)\\
=-\frac 1{t^2}\prod_j(1+\uf_j\cdot \uX+X_j\,t)
+\frac 1t \prod_j(1+\uf_j\cdot \uX+X_j\,t) \sum \frac {X_j}{1+\uf_j\cdot \uX+X_j\, t}
+\frac 1{t^2}\prod_j (1+\uf_j\cdot \uX)
\end{multline*}
evaluating at $t=1$ gives
\begin{multline*}
-(1+X_1+\cdots +X_n)^n + (1+X_1+\cdots +X_n)^{n-1}(X_1+\cdots+X_n)+\prod_j (1+\uf_j\cdot \uX)\\
= -(1+X_1+\cdots +X_n)^{n-1}+\prod_j (1+\uf_j\cdot \uX)\quad.
\end{multline*}
This shows that the left-hand side of~\eqref{eq:unlike} equals
\[
\frac 1{\prod_j (1+\uf_j\cdot \uX)}\cdot 
\left(\prod_j (1+\uf_j\cdot \uX) - (1+X_1+\cdots +X_n)^{n-1}\right)\quad,
\]
that is, the right-hand side, and concludes the proof of Proposition~\ref{nontri}.
\end{proof}



\end{document}